%% file: better-picture-hanging.tex
\documentclass[10pt,a4paper]{article}

\input{preamble.tex}

\usepackage{geometry}
\usepackage{amsthm}
\usepackage{graphicx}
\usepackage{hyperref}
\usepackage{cite}
\usepackage{xcolor}

\newtheorem{theorem}{Theorem}
\newtheorem{proposition}[theorem]{Proposition}
\newtheorem{lemma}[theorem]{Lemma}

\theoremstyle{definition}
\newtheorem{remark}[theorem]{Remark}

\title{The \kn{k}{n} picture-hanging puzzle:\\
  shorter solutions for small~$k$ and~$n-k$}

\author{Tom Verhoeff\thanks{Department of Mathematics \& Computer Science,
  Eindhoven University of Technology; \texttt{T.Verhoeff@tue.nl}}}

\date{\today}

\begin{document}

\maketitle

\begin{abstract}
The picture-hanging puzzle, popularized by Demaine et al.\ (2014),
asks for a way to wrap a wire around $n$ nails
such that the picture hangs as long as fewer than $k$ nails are removed,
but falls as soon as any $k$ are removed.
Solutions correspond to words in the free group $F_n$.
We give explicit, deterministic, polynomial-length constructions for two regimes:
\kn{2}{n} with word length at most $\tfrac{8}{3}n^{\log_2 6} - 4n^2$,
and \kn{(n-2)}{n} with word length $6n\log_2(n/2)$, both for $n$ a power of two.
These improve on Wästlund's quasi-polynomial deterministic construction
in their respective regimes.
We also report, via exhaustive computer search,
the exact minimum length of $16$ for the \kn{2}{4} puzzle,
attained by two structurally distinct solutions.
As an additional contribution,
we observe that the natural workshop realization with carabiners on a flat board
introduces an over/under ambiguity at every wire crossing;
a wrong choice can produce a Whitehead link,
which is topologically distinct from the intended commutator.
\end{abstract}

\section{Introduction}
\label{sec:intro}

The picture-hanging puzzle in its simplest form asks:
can one wrap a wire around two nails such that the picture falls
if either nail is removed?
The classical answer, popularized by Spivak~\cite{Spivak1997} in 1997 and
analyzed in detail by Demaine et al.~\cite{Demaine2014}, is yes:
the commutator $xyx^{-1}y^{-1}$ in the free group on two generators
corresponds to such a wiring.
In this paper we will write this commutator in additive notation
as $1 + 2 - 1 - 2$ (the rationale is explained in Section~\ref{sec:notation}).
Demaine et al.\ generalized this to arbitrary monotone Boolean functions,
with particular attention to the \emph{\kn{k}{n}} family:
the picture falls when at least $k$ of the $n$ nails are removed,
and hangs otherwise.
Their construction for \kn{k}{n} uses AKS sorting networks and
gives a word of polynomial length $O(n^{1\,561\,600})$ in the free group
--- polynomial, but with an astronomical exponent.

Wästlund~\cite{Waestlund2021} subsequently gave a simplified proof and
a much better polynomial bound of $O(n^{7.004})$ via a probabilistic construction,
together with an explicit deterministic divide-and-conquer construction
of quasi-polynomial length $n^{O(\log n)}$.
The explicit deterministic polynomial regime remained unexplored
for specific small values of $k$ or $n-k$.

We address this gap.
For \emph{fixed small} $k$ and \emph{fixed small} $n-k$
--- two natural regimes where the puzzle remains interesting ---
we give explicit deterministic constructions with low-degree polynomial length,
established by recurrences with closed-form solutions.
Combined with exhaustive computer search,
we also pin down exact minimum lengths for the smallest open cases.

\subsection*{Origin}

The present note arose from preparations for two public-outreach events:
a workshop for elementary and secondary school students
during a preparation day (March 2026),
organized by the Stichting Vierkant voor Wiskunde
for their summer mathematics camps,
and a mini-lecture at the MathFest in Utrecht (May 2026).
For practical reasons we used a wooden board with carabiners rather than nails:
a wire cannot accidentally slip off a carabiner the way it can off a smooth nail,
so participants could handle a partially completed construction
without the picture falling prematurely.
This pedagogical convenience comes at a cost: with carabiners on a flat board,
the wire's over/under choice at each crossing is no longer forced by gravity,
and incorrect choices produce topologically different links.
Section~\ref{sec:carabiners} discusses this.
For interactive experimentation, the author developed a Wolfram Demonstration~\cite{Verhoeff2026}.

While preparing materials,
the author began investigating the lengths of explicit \kn{k}{n} solutions.
The puzzle most accessible to participants is \kn{2}{n}
(``the picture is safe against the failure of any single nail'').
Iterative improvements
--- from the length-$80$ solution listed by Demaine et al.~for \kn{2}{4},
through successive reductions to lengths $68$, $58$, $54$, $24$, $20$, $18$ ---
led ultimately to the systematic constructions and the exact minimum reported here.

\subsection*{Contributions}

Our results are:

\begin{enumerate}
\item \emph{\kn{2}{n} (Theorem~\ref{thm:2-of-n}).}
  A general Demaine-perspective binary-splitting construction with Huffman commutator trees,
  specialized to $k = 2$, yields for $n = 2^i$, $i \ge 2$,
  a solution of length at most $\tfrac{8}{3} \cdot 6^i - 4 \cdot 4^i = \tfrac{8}{3}\,n^{\log_2 6} - 4n^2$.
  The exponent $\log_2 6 \approx 2.585$ is sub-cubic.
\item \emph{\kn{(n-2)}{n} (Theorem~\ref{thm:n-2-of-n}).}
  Wästlund's binary-split recursion,
  specialized to this regime and
  using optimal building blocks for the $k=1,2$ cases (Wästlund convention),
  yields, for $n = 2^i$, $i \ge 1$,
  a solution of length exactly $6(i-1) \cdot 2^i = 6n\log_2(n/2)$.
\item \emph{Exact minimum for \kn{2}{4} (Section~\ref{sec:search}).}
  Exhaustive symmetry-quotiented computer search establishes
  that the minimum length is $16$,
  attained by exactly two structurally distinct solutions.
\item \emph{Workshop realizations (Section~\ref{sec:carabiners}).}
  A discussion of the over/under ambiguity in carabiner-on-board realizations,
  with the simplest failure mode being a Whitehead link arising from a mis-wired \kn{1}{2} solution.
\end{enumerate}

\subsection*{Outline}

Section~\ref{sec:notation} fixes notation, sets out the algebra of expressions,
and treats the three easy cases.
Section~\ref{sec:binary-split} presents Demaine-perspective binary splitting
with Huffman commutator trees --- the general construction for \kn{k}{n},
then a closed-form length for the case \kn{2}{n}.
Section~\ref{sec:n-2-of-n} treats the \kn{(n-2)}{n} regime,
based on Wästlund-perspective binary splitting.
Section~\ref{sec:extension} presents an exponential one-step extension construction
that produces short explicit solutions for small $n$ and motivated the exhaustive search.
Section~\ref{sec:search} reports the exhaustive-search results,
including the proven minimum length $16$ for \kn{2}{4}.
Section~\ref{sec:carabiners} discusses the carabiner subtlety and
the Whitehead-link failure mode.
Section~\ref{sec:open} lists open questions.

\section{Notation and preliminaries}
\label{sec:notation}

\subsection*{An additive algebra of variables}

We work with formal expressions built from natural-number variables,
written $1, 2, 3, \ldots$ rather than $x_1, x_2, x_3, \ldots$.
The expressions use a binary operator $+$ and a unary inverse $-$.
Subject to the rules below, this is the free group on the variables;
we use the additive notation throughout
because we find it more accessible to a broader audience,
including secondary-school students and even upper-elementary pupils.

The expressions are governed by the following rules,
which the reader should keep in mind,
especially the contrast with ordinary integer arithmetic:
\begin{itemize}
\item \emph{Associativity:} $(x + y) + z \;=\; x + (y + z)$.
\item \emph{Non-commutativity:} in general $x + y \;\ne\; y + x$.
\item \emph{Neutral element:} we write $0$ for the empty expression,
  and $0 + x = x + 0 = x$.
\item \emph{Inverse:} $x + (-x) \;=\; 0 \;=\; (-x) + x$.
\item \emph{Double negation:} $-(-x) = x$.
\item \emph{Inversion of a sum:} $-(x + y) = -y + (-x) = -y - x$
  (note the order reversal).
\end{itemize}
We allow ourselves the shorthand $x + x = 2x$ and similar,
but the reader must keep in mind that, due to non-commutativity,
$p(x + y) \ne px + py$ in general.

The \emph{length} of an expression is the number of variables-with-sign
in its fully reduced form (where reduction means repeatedly applying $x + (-x) = 0$
and absorbing the resulting $0$).
A nonzero expression has positive length; the only zero expression has length $0$.
The \emph{support} of an expression is the set of variables that occur in it;
two expressions have \emph{disjoint supports} when they share no variable.

\subsection*{The picture-hanging puzzle}

A solution $w$ to the \emph{\kn{k}{n} picture-hanging puzzle} (Demaine's convention)
is an expression in variables $1, \ldots, n$ such that:
\begin{itemize}
\item[(a)] $w$ is nonzero;
\item[(b)] setting any $k$ of the variables $1, \ldots, n$ to $0$ reduces $w$ to $0$;
\item[(c)] for every set of fewer than $k$ variables,
  setting those variables to $0$ leaves $w$ nonzero.
\end{itemize}
Condition (b) corresponds to the picture falling when any $k$ nails are removed;
condition (c) corresponds to the picture still hanging when fewer than $k$ are removed.
We write $\Hang{k}{n}$ for any solution to the \kn{k}{n} puzzle,
and $\Len{k}(n)$ for the length of a particular solution under discussion.

\subsection*{Specifications}

It helps to name the object a puzzle asks us to realize.
A \emph{specification} on a set of nails $V$ is a monotone Boolean function
\[
  f : 2^V \to \{\mathsf{hang}, \mathsf{fall}\}, \qquad f(V) = \mathsf{fall},
\]
where $f(S)$ is the outcome when exactly the nails in $S$ are removed
(Demaine's convention).
We order $\mathsf{hang} < \mathsf{fall}$ and read $\vee$ and $\wedge$
as $\max$ and $\min$,
so that $\mathsf{fall} \vee x = \mathsf{fall}$
and $\mathsf{hang} \wedge x = \mathsf{hang}$;
\emph{monotone} then means $S \subseteq S' \Rightarrow f(S) \le f(S')$,
that is, removing more nails never restores the picture.
We say $f$ \emph{falls} at $S$ when $f(S) = \mathsf{fall}$,
and \emph{hangs} at $S$ otherwise.
An expression $h$ on $V$ \emph{solves} $f$ if
\[
  h|_S = 0 \iff f(S) = \mathsf{fall}
  \qquad \text{for every } S \subseteq V,
\]
where $h|_S$ denotes $h$ with the nails in $S$ removed (set to $0$).

The condition $f(V) = \mathsf{fall}$ is forced, not chosen:
removing every nail reduces any expression on~$V$ to~$0$,
so no expression can solve a function that hangs at~$V$.
By monotonicity it excludes a single function,
the constant $\mathsf{hang}$
--- which never falls, even when all nails are removed ---
the lone monotone specification that no wire can realize.
The opposite constant, $\mathsf{fall}$, is solved by the empty expression~$0$
(nothing is hung, so the picture has already fallen).
Every other monotone function with $f(V) = \mathsf{fall}$ is realizable too:
Demaine et al.~\cite{Demaine2014} and Wästlund~\cite{Waestlund2021}
each construct an expression that solves it.
(Wästlund excludes both constants as trivial;
we keep $\mathsf{fall}$, since the empty expression $0$ solves it.)
Our concern is not realizability but the search for \emph{shorter} solutions.

The puzzle of interest is itself a specification.
Write $\kof{k}{V}$ for the specification on $V$ with
\begin{equation}
\label{eq:kof}
  (\kof{k}{V})(S) = \mathsf{fall} \iff |S| \ge k.
\end{equation}
A solution to the \kn{k}{n} puzzle is then exactly an expression
solving $\kof{k}{V}$ on $V = \{1, \ldots, n\}$:
conditions (a)--(c) say precisely that $w|_S = 0$ iff $|S| \ge k$.

The conditions are monotone:
removing more nails cannot un-fall the picture.
If the picture falls when a set $S$ of nails is removed,
it also falls when any larger set $S' \supseteq S$ is removed,
since an expression already equal to $0$ remains $0$ when more variables are set to $0$.
Equivalently,
if the picture hangs with a set $S'$ of nails removed,
it also hangs with any subset $S \subseteq S'$ removed.

Monotonicity lets a correctness proof focus on the \emph{essential} removals
--- those whose required outcome does not already follow from another removal.
A required fall is essential only when no smaller removal already falls,
and a required hang only when no larger removal already hangs.
For $\kof{k}{V}$ the essential removals are therefore exactly
the sets of $k$ nails (which must fall) and the sets of $k-1$ nails (which must hang):
removing fewer than $k-1$ hangs by monotonicity from the $(k-1)$ case,
and removing more than $k$ falls by monotonicity from the $k$ case.
It thus suffices to check that every removal of $k$ nails makes the picture fall
and every removal of $k-1$ nails leaves it hanging.

For example,
the expression $1 + 2 - 1 - 2$ has length $4$.
It is nonzero, since no reduction rule applies:
no two adjacent terms are mutual inverses.
Setting $1$ to $0$ reduces it to $2 - 2 = 0$;
setting $2$ to $0$ reduces it to $1 - 1 = 0$.
Hence it solves the \kn{1}{2} puzzle:
it hangs as it stands, yet removing either nail (variable)
causes the picture (expression) to fall (vanish).

In group-theoretic terms,
$1 + 2 - 1 - 2$ is the commutator $[x_1, x_2]$ in the free group on two generators.
We define the \emph{commutator} of two expressions $a$ and $b$ as
\[
  \Comm{a}{b} \;=\; (a + b) - (b + a) \;=\; a + b - a - b,
\]
where the second equality uses the rule $-(b + a) = -a - b$.
The commutator thus \emph{measures non-commutativity}:
it is the difference between the two possible orders of summing $a$ and $b$.

Three key properties of the commutator follow immediately from this definition:

\begin{itemize}
\item \emph{Antisymmetry:}
  $-\Comm{a}{b} = \Comm{b}{a}$, because $-(a+b-a-b) = b+a-b-a$.
\item \emph{Vanishing condition:}
  $\Comm{a}{b} = 0$ iff $a + b = b + a$, that is, iff $a$ and $b$ commute under addition.
  In the free group,
  this happens exactly when $a$ and $b$ are both multiples of a common element
  --- in group theory, a fact about \emph{centralizers}.
\item \emph{Special cases that vanish:}
  $\Comm{0}{b} = \Comm{a}{0} = \Comm{a}{a} = \Comm{a}{-a} = 0$,
  and more generally $\Comm{pa}{qa} = 0$ for any integers $p, q$.
\end{itemize}

We will use commutators recursively:
the arguments $a$ and $b$ may themselves be expressions built from
variables, sums, and commutators.
The vanishing condition then becomes a useful test, which we phrase as follows.
Say that two expressions \emph{solve the same problem}
when they vanish under exactly the same removals.
The inverse $-h$ solves the same problem as $h$;
so does any nonzero multiple $p\,h$ (that is, $h$ repeated $p$ times, $p \ne 0$),
since a nonzero multiple of a nonzero expression is again nonzero
(in group theory, the free group is \emph{torsion-free});
and so does any \emph{conjugate} $g + h - g$ (a copy of $h$ wrapped in $g$ and $-g$),
since it too vanishes exactly when $h$ does.
Hence two multiples of a common element solve the same problem.
Reading the vanishing conditions through this lens, for nonzero $a$ and $b$:
$a + b = 0$ forces $b = -a$, and $\Comm{a}{b} = 0$ forces $a$ and $b$ to be common multiples
--- in either case $a$ and $b$ solve the same problem.
Contrapositively, expressions that solve \emph{different} problems
satisfy $a + b \ne 0$ and $\Comm{a}{b} \ne 0$;
in particular this holds when they have disjoint, nonempty supports.
This gives only non-vanishing; that $+$ and $\Comm{}{}$ actually realize $\wedge$ and $\vee$
needs more, as we discuss in Section~\ref{sec:binary-split}.

\subsection*{Demaine's and Wästlund's conventions}

Wästlund~\cite{Waestlund2021} adopts the opposite convention to Demaine:
his \kn{k}{n} means the picture \emph{hangs} iff at least $k$ nails remain.
The translation is:
\[
  \text{Demaine's \kn{k}{n}} \;=\; \text{Wästlund's \kn{(n-k+1)}{n}.}
\]
Throughout this paper we use Demaine's convention
unless we explicitly switch to Wästlund's in Section~\ref{sec:n-2-of-n},
where it is structurally more convenient.

\subsection*{Three easy cases}

Three cases of \kn{k}{n} admit short, optimal, and explicit solutions.

\begin{proposition}[$k=n$]
\label{prop:k-equals-n}
The minimum length of a solution to the \kn{n}{n} puzzle is exactly $n$,
attained by $1 + 2 + \cdots + n$.
\end{proposition}

\begin{proof}
This expression is $0$ iff all variables are set to $0$,
and remains nonzero under any proper subset of removals
(some surviving variable is left, and the reduced expression contains it).
Any solution must mention every variable, so length $\ge n$.
\end{proof}

\begin{proposition}[$k=n-1$]
\label{prop:k-equals-n-1}
The minimum length of a solution to the \kn{(n-1)}{n} puzzle is exactly $2n$,
attained by
\[
  1 + 2 + \cdots + n - 1 - 2 - \cdots - n.
\]
\end{proposition}

\begin{proof}
Let $w = 1 + 2 + \cdots + n - 1 - 2 - \cdots - n$.

\emph{Nonzero:} $w$ is freely reduced (no two adjacent terms are mutual inverses),
so $w \ne 0$.

\emph{Fall under $n-1$ removals:} When $n-1$ variables are removed,
exactly one variable $j$ survives, and the expression reduces to $j - j = 0$.

\emph{Hang under fewer removals:} When exactly $n-2$ variables are removed,
the two survivors $a < b$ give $a + b - a - b = \Comm{a}{b} \ne 0$ (since $a \ne b$,
they do not commute).
By monotonicity, the picture also hangs under any fewer removals.

\emph{Lower bound $2n$:} In any solution,
fixing variable $j$ and setting all other $n-1$ variables to $0$
removes $n-1 \ge k = n-1$ variables, so the resulting expression must be $0$.
Hence the net exponent of $j$ --- its number of $+j$ occurrences minus its number of $-j$ ---
is $0$ in any solution, forcing $j$ to appear at least twice.
Since this holds for every $j$, length $\ge 2n$.
\end{proof}

\begin{proposition}[$k=1$]
\label{prop:k-equals-1}
The minimum length of a solution to the \kn{1}{n} puzzle is $\Theta(n^2)$.
For $n$ a power of two, length $n^2$ is attained by the balanced commutator tree,
recursively defined by
\[
  \Hang{1}{1} = 1, \qquad
  \Hang{1}{2n} = \Comm{\Hang{1}{n}\text{ on }\{1,\ldots,n\}}{\Hang{1}{n}\text{ on }\{n+1,\ldots,2n\}}.
\]
The lower bound $\Omega(n^2)$ is due to Gartside and Greenwood~\cite{GartsideGreenwood2007};
the upper-bound construction goes back at least to \cite{Taylor2002}.
\end{proposition}

\begin{proof}[Sketch]
The recursion $L(2n) = 4 L(n)$ with $L(1) = 1$ gives $L(2^i) = 4^i = (2^i)^2$.
Correctness follows by induction:
a commutator $\Comm{A}{B}$ with $A, B$ on disjoint variable sets vanishes
iff $A = 0$ or $B = 0$,
and inductively each $\Hang{1}{n}$ vanishes iff some variable in its support is removed.
Hence $\Comm{\Hang{1}{n}\text{-on-}L}{\Hang{1}{n}\text{-on-}R}$ vanishes
iff some variable on either side is removed
--- i.e., iff any of the $2n$ variables is removed.
\end{proof}

The construction of Proposition~\ref{prop:k-equals-1} is, in fact,
the simplest case of Demaine-perspective binary splitting
(Section~\ref{sec:binary-split}):
for $k=1$, the disjunction at each split has just two terms
(``some variable removed in $L$'' or ``some variable removed in $R$''),
giving a 2-leaf commutator that is the balanced tree.

The next case in the co-rank family, \kn{(n-2)}{n}, is genuinely non-trivial;
it is the subject of Section~\ref{sec:n-2-of-n}.

\section{Demaine-perspective binary splitting}
\label{sec:binary-split}

In this section we present Demaine-perspective binary splitting,
a deterministic and explicit divide-and-conquer construction for the \kn{k}{n} puzzle.
We first describe the construction for general~$k$,
and then prove a closed-form length expression for the case $k = 2$ on powers of two.
The correctness of the construction,
for all $k$, is taken up in Section~\ref{sec:correctness}.

\subsection{The general construction}
\label{sec:general}

Demaine-perspective binary splitting builds a \kn{k}{n} solution from solutions of smaller puzzles.
Choose any split of the $n$ nails into two nonempty parts:
$L$ of size $n_1$ and $R$ of size $n_2$, with $n = n_1 + n_2$ and $n_1, n_2 \ge 1$.
The construction is correct for \emph{every} such split;
it is most efficient when $n_1$ and $n_2$ are as equal as possible,
and for the length analysis we will take $n_1 = \lceil n/2 \rceil$ and $n_2 = \lfloor n/2 \rfloor$.
In Demaine's convention the picture falls as soon as $k$ nails are removed;
we classify a removal by how it splits across the two parts.
If $j$ of the removed nails lie in $L$, the remaining removals lie in $R$,
and ``at least $k$ removed'' becomes ``at least $j$ removed in $L$ and at least $k-j$ removed in $R$.''
Hence
\begin{equation}
\label{eq:disjunction}
  \text{\kn{k}{n} falls} \iff \bigvee_{j} \Bigl[ (\text{$\ge j$ removed in $L$}) \;\wedge\; (\text{$\ge k-j$ removed in $R$}) \Bigr],
\end{equation}
where $j$ ranges over the feasible values $\max(0,\,k-n_2) \le j \le \min(k,\,n_1)$
--- those for which neither part is asked to lose more nails than it contains.

The two logical operations have direct algebraic counterparts,
provided the parts they combine have disjoint supports.

\begin{lemma}[Disjoint combination]
\label{lem:disjoint}
Let $h_1$ solve a specification $f_1$ on a nail set $V_1$,
and $h_2$ solve a specification $f_2$ on a disjoint set $V_2$.
Then, over $V_1 \cup V_2$,
\[
  h_1 + h_2 \ \text{ solves }\ f_1 \wedge f_2,
  \qquad
  \Comm{h_1}{h_2} \ \text{ solves }\ f_1 \vee f_2.
\]
The specifications are arbitrary:
the proof uses only the disjointness of the supports,
not any special structure of $f_1$ or $f_2$.
\end{lemma}

\begin{proof}
A removal $S$ acts on each part separately
(in group theory it is a \emph{homomorphism}, so it respects $+$ and $\Comm{}{}$):
$(h_1 + h_2)|_S = h_1|_S + h_2|_S$ and $\Comm{h_1}{h_2}|_S = \Comm{h_1|_S}{h_2|_S}$,
with $h_i|_S$ supported in $V_i$.
\emph{Sum:} $h_1|_S + h_2|_S = 0$ would force $h_2|_S = -h_1|_S$,
an element common to $V_1$ and $V_2$, hence $0$
(in group theory, the subexpressions on $V_1$ and $V_2$ lie in \emph{free factors}, meeting only in $0$);
so the sum vanishes iff $h_1|_S = h_2|_S = 0$,
that is, iff $f_1$ and $f_2$ both fall at $S$;
hence $h_1 + h_2$ solves $f_1 \wedge f_2$.
\emph{Commutator:} if either part vanishes, so does the commutator;
if both are nonzero, they have disjoint nonempty supports
and hence are not multiples of a common element,
so they do not commute and $\Comm{h_1|_S}{h_2|_S} \ne 0$.
Thus the commutator vanishes iff $f_1$ or $f_2$ falls at $S$;
hence $\Comm{h_1}{h_2}$ solves $f_1 \vee f_2$.
\end{proof}

The disjointness hypothesis is essential: without it, both statements can fail.
For the sum, overlapping supports may cancel,
so that $h_1 + h_2$ vanishes although neither part does.
The commutator is subtler.
Consider the puzzle that falls exactly when nail~$3$ is removed together with nail~$1$ or nail~$2$,
the condition $(1 \vee 2) \wedge 3$.
A correct solution is $\Comm{1}{2} + 3$:
the disjunction $1 \vee 2$ on the disjoint support $\{1,2\}$, conjoined with~$3$.
Distributing $\wedge$ over~$\vee$ rewrites the condition as $(1 \wedge 3) \vee (2 \wedge 3)$,
inviting the ``solution'' $\Comm{1+3}{\,2+3}$.
But its arguments share nail~$3$:
removing $1$ and $2$ while leaving $3$ sends $1+3 \mapsto 3$ and $2+3 \mapsto 3$,
so the commutator collapses to $\Comm{3}{3} = 0$ and the picture wrongly falls.
The shared nail has made the two arguments multiples of a common element,
exactly where the commutator vanishes.
The disjuncts of our construction overlap in just this way,
so Lemma~\ref{lem:disjoint} does not cover the commutator tree that combines them;
that case is the substance of the correctness proof (Section~\ref{sec:correctness}).

The $j$-th disjunct is then the subexpression
\[
  D_j = \begin{cases}
    \Hang{k}{R}, & j = 0,\\[2pt]
    \Hang{j}{L} + \Hang{(k-j)}{R}, & 0 < j < k,\\[2pt]
    \Hang{k}{L}, & j = k,
  \end{cases}
\]
where $\Hang{k}{V}$ denotes any solution of the specification $\kof{k}{V}$;
the cases $j = 0$ and $j = k$ drop the vacuous condition ``$\ge 0$ removed,''
which always holds.
These subexpressions, one per feasible $j$,
are then combined into a single expression by a \emph{commutator tree},
realizing their disjunction.
When $n_1, n_2 \ge k$ all $k+1$ values of $j$ are feasible;
otherwise the range is shorter.

The recursion bottoms out at the easy cases of Section~\ref{sec:notation}.
Notably, for $k = 1$ every split has exactly two disjuncts
--- ``something removed in $L$'' or ``something removed in $R$'' ---
so the construction reduces to the two-leaf commutator
$\Comm{\Hang{1}{L}}{\Hang{1}{R}}$;
with balanced splits throughout,
this is the commutator tree of Proposition~\ref{prop:k-equals-1}.

\paragraph{Huffman placement.}
A commutator tree over the disjuncts may take many shapes,
all realizing the same disjunction but with different total lengths.
Since each commutator $\Comm{A}{B} = A + B - A - B$ duplicates both of its arguments,
a leaf placed at depth $d$ contributes $2^d$ times its own length to the total.
Minimizing the total length is therefore a Huffman problem on the leaf lengths $|D_j|$:
one repeatedly combines the two shortest current subexpressions into a commutator,
which places the longest leaves nearest the root.
We use Huffman placement at every split.

\paragraph{Correctness and length.}
Two points remain.
First, that the construction actually solves the \kn{k}{n} puzzle.
Each conjunction here --- a cross term $A_j + B_{k-j}$ on the disjoint halves $L, R$ ---
is sound on its own by Lemma~\ref{lem:disjoint}.
The commutator tree is not: its leaves, the disjuncts, share supports
--- for instance $D_0$ and $D_1$ both constrain $R$ ---
so Lemma~\ref{lem:disjoint} does not apply to it,
and correctness of the overall disjunction needs the inductive argument of Section~\ref{sec:correctness}.
Second,
the total length is governed by a recurrence determined by the Huffman placement of the leaves.
We solve this recurrence in closed form for the case $k = 2$
with balanced splits in the remainder of this section;
the general case is left open (Section~\ref{sec:open}).

\subsection{The case $k = 2$}
\label{sec:bs-construction}

We now specialize to $k = 2$ and to balanced splits,
taking $n$ even and $L = \{1, \ldots, m\}$, $R = \{m+1, \ldots, n\}$ with $m = n/2$.
The decomposition has three disjuncts:
``both removed in $L$,'' ``both removed in $R$,'' and ``one removed in each half.''
The first two are \kn{2}{m} conditions on $L$ and on $R$;
the third is the conjunction of a \kn{1}{m} condition on $L$ with a \kn{1}{m} condition on $R$,
realized by their sum.
The three resulting subexpressions are combined into a three-leaf commutator tree.
With three leaves the only shape is $\Comm{X}{\Comm{Y}{Z}}$,
one leaf at depth $1$ and two at depth $2$;
Huffman placement puts the longest of the three at depth $1$ (a single doubling)
and the two shorter at depth $2$ (a double doubling).

\subsection{Correctness of the construction}
\label{sec:correctness}

We show that the construction of Section~\ref{sec:general}
solves $\kof{k}{V}$ on $V = \{1, \ldots, n\}$,
for every $k$ and every split; explicitly,
\begin{equation}
\label{eq:char}
  w|_S = 0 \iff |S| \ge k \qquad \text{for all } S,
\end{equation}
where $w|_S$ is $w$ with the nails in $S$ removed (Section~\ref{sec:notation}).
The construction realizes the logical identity
\begin{equation}
\label{eq:split}
  \kof{k}{V} \;=\; \bigvee_j \bigl( \kof{j}{L} \wedge \kof{(k-j)}{R} \bigr)
\end{equation}
of Section~\ref{sec:general}:
it builds each cross term with a sum
and combines the terms with a commutator tree.
The sums are sound by Lemma~\ref{lem:disjoint}, since $L$ and $R$ are disjoint.
The one step still in doubt
--- already flagged by the $(1 \vee 2) \wedge 3$ failure of
Section~\ref{sec:general} ---
is that the commutator tree realizes the disjunction,
even though its leaves share supports.
We isolate the exact obstruction,
phrase it as a condition on specifications,
and then discharge that condition for the construction.

Both algebraic combiners can vanish spuriously when supports overlap.
A removal acts term by term, so
$(h_1 + h_2)|_S = h_1|_S + h_2|_S$
and $\Comm{h_1}{h_2}|_S = \Comm{h_1|_S}{h_2|_S}$.
With both parts nonzero, a sum vanishes exactly when $h_2|_S = -h_1|_S$,
and a commutator exactly when $h_1|_S$ and $h_2|_S$ commute;
either way (Section~\ref{sec:notation}) the two are multiples of a common
element, and so solve the same problem.
A single condition rules both out.
Say that specifications $g_1$ and $g_2$ \emph{separate above} a set $S$
if $g_1(S') \ne g_2(S')$ for some $S' \supseteq S$.

\begin{lemma}[Combination]
\label{lem:combination}
Let $h_1$ solve $g_1$ and $h_2$ solve $g_2$,
and suppose $g_1$ and $g_2$ separate above every $S$ at which both hang.
Then $h_1 + h_2$ solves $g_1 \wedge g_2$,
and $\Comm{h_1}{h_2}$ solves $g_1 \vee g_2$.
On disjoint supports the hypothesis is automatic,
so this generalizes Lemma~\ref{lem:disjoint}.
\end{lemma}

\begin{proof}
Each combiner has an automatic direction:
if $g_1 \wedge g_2$ falls at $S$ then $h_1|_S = h_2|_S = 0$ and the sum vanishes,
while if $g_1 \vee g_2$ falls at $S$ then some $h_i|_S = 0$
and the commutator vanishes.
For the converses, suppose a combiner vanishes where its target hangs.
The sum's target hangs where some $g_i$ hangs;
were only one part nonzero the sum could not vanish, so both $h_i|_S$ are nonzero.
The commutator's target hangs only where both $g_i$ hang,
so again both $h_i|_S$ are nonzero.
A spurious vanishing now makes $h_1|_S$ and $h_2|_S$ multiples of a common
element, hence (Section~\ref{sec:notation}) solutions of the same problem:
they vanish under the same $S' \supseteq S$, so $g_1$ and $g_2$ agree above $S$.
But both hang at $S$, so they separate above $S$ --- a contradiction.
On disjoint supports $V_1, V_2$ the hypothesis is free:
at any $S$ where both hang, additionally removing all of $V_1$
makes $g_1$ fall (as $g_1(V_1) = \mathsf{fall}$) but leaves $g_2$ unchanged.
\end{proof}

The construction nests many commutators, not one.
Lemma~\ref{lem:combination} is a sufficient test on specifications;
for a whole tree the next theorem makes the criterion exact, and milder:
a commutator may harmlessly collapse at a removal where the target already
falls, so only the removals where the target hangs need care.

\begin{theorem}[Commutator trees as disjunctions]
\label{thm:tree}
Let $T$ be a commutator tree whose leaves solve specifications
$g_1, \ldots, g_m$, and let $f = g_1 \vee \cdots \vee g_m$.
Then $T$ solves $f$ if and only if no internal node of $T$ vanishes
at an $S$ where $f$ hangs.
In particular, $T$ solves $f$ if at every such $S$
the specifications solved by the two subtrees at each internal node
separate above it.
\end{theorem}

\begin{proof}
The minterms of $f$ are exactly those of the $g_i$, each realized by a leaf.
If $f$ falls at $S$, some leaf vanishes; a commutator with a vanishing argument
vanishes, and this propagates to the root, so $T|_S = 0$ whatever the
internal nodes do.
Hence $T$ solves $f$ precisely when, in addition, $T|_S \ne 0$ wherever
$f$ hangs.
A vanishing internal node propagates to the root,
and a nonvanishing root is itself the topmost internal node,
so $T|_S \ne 0$ at an $f$-hang $S$ iff no internal node vanishes there;
this is the equivalence.
For the sufficient condition: at an $f$-hang $S$ no leaf is covered,
so every leaf hangs and, inductively, the two subtrees at each node are nonzero;
if their specifications separate above $S$ they cannot commute
(Lemma~\ref{lem:combination}), so the node is nonzero.
\end{proof}

There is no counterpart of Theorem~\ref{thm:tree} for sums.
A commutator tree falls as soon as any one leaf is covered,
so a collapse at a removal where the tree should fall does no harm;
a sum falls only when \emph{all} its terms vanish, and has no such slack.
An overlapping sum must therefore avoid cancellation
at every removal where it hangs:
the disjunctive side is the forgiving one.

For the construction, fix the split $V = L \cup R$ of Section~\ref{sec:general},
with $|L| = n_1$, $|R| = n_2$;
for a set $S$ write $\ell = |S \cap L|$ and $r = |S \cap R|$.
The leaves are the disjuncts:
$D_j$ solves the cross term $\kof{j}{L} \wedge \kof{(k-j)}{R}$
(Lemma~\ref{lem:disjoint}, disjoint halves),
which depends on $S$ only through $(\ell, r)$
and falls on the quadrant $\ell \ge j$, $r \ge k - j$ with corner $(j, k-j)$.
A subtree over a leaf-set $J$ should compute
\begin{equation}
\label{eq:fJ}
  f_J \;:=\; \bigvee_{j \in J} \bigl( \kof{j}{L} \wedge \kof{(k-j)}{R} \bigr),
\end{equation}
a union of quadrants.
Call a specification on $L \cup R$ a \emph{staircase}
if it depends on $S$ only through $(\ell, r)$ and is monotone
--- an up-set in the $(\ell, r)$ grid, fixed by its minimal corners.
Each $f_J$ is a staircase, and staircases over disjoint index sets separate.
Figure~\ref{fig:staircase} illustrates the staircase for $k = 3$,
$n_1 = n_2 = 3$.

\begin{figure}[ht]
  \centering
  \includegraphics[width=0.55\linewidth]{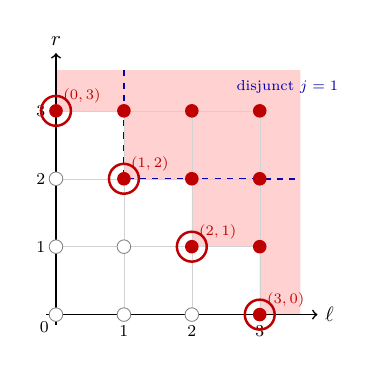}
  \caption{Staircase fall region of $\kof{k}{V}$ for $k = 3$
    and $n_1 = n_2 = 3$.
    Filled points are fall, open points hang;
    the four ringed corners $(j, k - j)$ mark the disjuncts
    $j = 0, 1, 2, 3$.
    Within the full feasible range of $j$,
    consecutive corners always differ by $(+1, -1)$
    --- a unit diagonal step.
    A sub-staircase $f_J$ over a non-contiguous $J$
    (for instance $J = \{0, 2\}$, with corners $(0,3)$ and $(2,1)$)
    has a larger step.
    The blue dashes outline the quadrant of disjunct $j = 1$.}
  \label{fig:staircase}
\end{figure}

\begin{lemma}[Staircase separation]
\label{lem:frontier}
Let $J_1, J_2$ be disjoint sets of feasible indices,
and let $S$ be a set at which neither $f_{J_1}$ nor $f_{J_2}$ falls.
Then $f_{J_1}$ and $f_{J_2}$ separate above $S$.
\end{lemma}

\begin{proof}
Parametrize the supersets $S' \supseteq S$ by the numbers $a, b \ge 0$
of \emph{additional} nails removed in $L$ and in $R$.
Disjunct $j$ falls at $S'$ iff $\ell + a \ge j$ and $r + b \ge k - j$,
that is, iff $a \ge \alpha_j$ and $b \ge \beta_j$,
where $\alpha_j = \max(0,\, j - \ell)$ and $\beta_j = \max(0,\, k - j - r)$.
So, over these supersets, $f_J$ is the union of the quadrants
with corners $(\alpha_j, \beta_j)$, $j \in J$,
and its minimal corners are among these.

The corner map $j \mapsto (\alpha_j, \beta_j)$ is injective (one-to-one)
on the indices hanging at $S$.
Indeed $\alpha_j$ is nondecreasing and $\beta_j$ nonincreasing in $j$;
if $\alpha_j = \alpha_{j'}$ with $j \ne j'$, both are $0$, so $j, j' \le \ell$;
if moreover $\beta_j = \beta_{j'}$, both are $0$, so $j, j' \ge k - r$
--- but then $\alpha_j = \beta_j = 0$, i.e.\ disjunct $j$ falls at $S$,
which is excluded.

Each $f_{J_i}$ hangs at $S$ (the corner $(0,0)$ is absent)
but falls once enough further nails are removed,
so its quadrant union has at least one minimal corner.
Were $f_{J_1}$ and $f_{J_2}$ to agree on all $S' \supseteq S$,
the two staircases would coincide and share their minimal corners.
But every index in $J_1 \cup J_2$ hangs at $S$ (else its $f_{J_i}$ would fall),
and distinct indices hanging at $S$ have distinct corners,
so a shared corner would come from one index lying in both $J_1$ and $J_2$,
contradicting $J_1 \cap J_2 = \emptyset$.
Hence $f_{J_1}$ and $f_{J_2}$ separate above $S$.
\end{proof}

\begin{theorem}
\label{thm:correctness}
For every $k$ with $1 \le k \le n$ and every split,
the construction of Section~\ref{sec:general} produces a valid \kn{k}{n} solution.
\end{theorem}

\begin{proof}
Induction on $n$; for $n = 1$ the single nail solves \kn{1}{1}.
For $n \ge 2$, by induction the subproblems are solved correctly,
so $D_j = \Hang{j}{L} + \Hang{(k-j)}{R}$ combines solutions on the disjoint
halves and hence solves $f_{\{j\}}$ by Lemma~\ref{lem:disjoint}
(the boundary cases $D_0 = \Hang{k}{R}$ and $D_k = \Hang{k}{L}$ likewise).
The tree thus has leaves solving the $f_{\{j\}}$,
with target $\bigvee_j f_{\{j\}} = \kof{k}{V}$ by~\eqref{eq:split}.
Wherever $\kof{k}{V}$ hangs we have $\ell + r < k$,
so no disjunct falls (each needs $\ell + r \ge k$)
and every staircase $f_J$ hangs;
hence at each internal node the two child staircases separate above $S$
(Lemma~\ref{lem:frontier}),
and Theorem~\ref{thm:tree} gives that the tree solves $\kof{k}{V}$.
Separation in fact holds wherever both children hang,
so each subtree even computes its own $f_J$.
\end{proof}

\begin{remark}[Demaine's \kn{2}{4} solution]
\label{rem:demaine}
Demaine et al.~\cite{Demaine2014} give the \kn{2}{4} solution
\[
  \Comm{\Comm{1+2}{\Comm{1+3}{1+4}}}{\Comm{2+3}{\Comm{2+4}{3+4}}},
\]
a commutator tree whose six leaves are the sums $i + j$, one per pair of nails,
each solving the conjunction that nails $i$ and $j$ are both removed
(Figure~\ref{fig:demaine-tree}).
That it falls when any two nails are removed is immediate:
the matching leaf vanishes and zeros the tree.
That it hangs under one nail or none follows from Theorem~\ref{thm:tree}:
the disjunction of the six leaves' minterms is $\kof{2}{V}$,
and the theorem reduces correctness,
by the essential-removals reduction (Section~\ref{sec:notation}),
to checking that no internal node vanishes at any of the four single-nail
removals --- a small manual verification.
The $\binom{4}{2} = 6$ two-nail removals need no separate check
since each removes the matching leaf and so zeros the tree.
The same reduction applies to any commutator tree built from the six pair-sums,
and more generally to \kn{2}{n} for any $n$ by taking the $\binom{n}{2}$
pair-sums as leaves, leaving the $n$ single-nail removals to check.

The subtree $\Comm{1+3}{1+4}$ collapses at the removal $\{3,4\}$:
both arguments reduce to nail~$1$ and so commute,
so the subtree solves $\kof{2}{\{1,3,4\}}$, not just the disjunction of its two leaves.
This is a harmless ``don't-care''
since $\{3,4\}$ is a removal at which $\kof{2}{V}$ already falls;
it also makes the explicit leaf $3+4$ redundant,
opening the way to a shorter solution (Section~\ref{sec:concrete-2-of-4}).
\end{remark}

\begin{figure}[ht]
  \centering
  \includegraphics[width=0.85\linewidth]{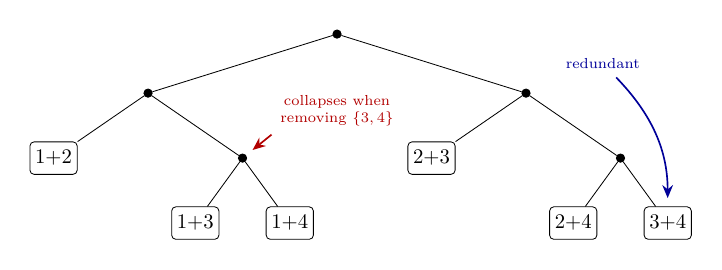}
  \caption{Demaine's commutator tree for the \kn{2}{4} puzzle.
    Each leaf $i + j$ is the pair conjunction
    ``nails $i$ and $j$ are both removed.''
    The inner node above $1{+}3$ and $1{+}4$ collapses when nails $3$ and
    $4$ are removed: both arguments then reduce to nail $1$,
    which commutes with itself.
    This collapse is harmless because $\{3,4\}$ is a removal at which
    $\kof{2}{V}$ already falls;
    it also makes the explicit leaf $3+4$ redundant,
    since its minterm $\{3,4\}$ is supplied by the collapse.}
  \label{fig:demaine-tree}
\end{figure}

\subsection{Recurrence and closed form}

Let $\Len{2}(n)$ denote the length of the construction's word for \kn{2}{n} on $n = 2^i$ nails,
counted before any free reduction at the commutator junctions, with $m = n/2$ as above.
Because reduction can only shorten the word,
$\Len{2}(n)$ is an upper bound on the length of the solution actually produced
(see the remark after Theorem~\ref{thm:2-of-n}).
The three disjunction subexpressions are:
\begin{itemize}
\item ``both removed in $L$''
  --- a \kn{2}{m} subproblem on $L$, of length $\Len{2}(m)$;
\item ``both removed in $R$''
  --- a \kn{2}{m} subproblem on $R$, of length $\Len{2}(m)$;
\item the cross term ``one in each''
  --- the sum of a \kn{1}{m} solution on $L$ and a \kn{1}{m} solution on $R$,
  of length $2\,\Len{1}(m) = 2m^2 = \tfrac{1}{2}n^2$ by Proposition~\ref{prop:k-equals-1}.
\end{itemize}

These three subexpressions are the leaves of a commutator tree.
Recall from Section~\ref{sec:general} that a leaf at depth $d$ contributes $2^d$ times its length.
With three leaves there is one leaf at depth $1$ and two at depth $2$,
so the total length is $4W - 2\ell$,
where $W$ is the sum of the three leaf lengths and $\ell$ is the length of the leaf at the root.
The length is minimized by placing the \emph{longest} leaf at the root --- the Huffman rule ---
equivalently merging the two shortest leaves into the depth-$2$ commutator.

Which leaf is longest depends on $n$.
At $n = 4$ (so $m = 2$) the cross term, of length $2m^2 = 8$,
is longer than each same-half subproblem ($\Len{2}(2) = 2$),
so it sits at the root and gives the base value
\[
  \Len{2}(4) \;=\; 2 \cdot 8 + 4 \cdot 2 + 4 \cdot 2 \;=\; 32.
\]
For every $n \ge 8$ the cross term is instead one of the two \emph{shortest} leaves,
so the optimal tree merges it with one same-half subproblem at depth $2$ and
places the other same-half subproblem at the root:
\[
  \Len{2}(n) \;=\; \underbrace{2\,\Len{2}(m)}_{\text{root}} \;+\; \underbrace{4\,\Len{2}(m) + 4 \cdot 2m^2}_{\text{depth }2} \;=\; 6\,\Len{2}(m) + 8m^2 \;=\; 6\,\Len{2}(n/2) + 2n^2.
\]
That the cross term really is among the two shortest leaves for $n \ge 8$
amounts to $\Len{2}(m) \ge 2m^2$ for $m \ge 4$, which follows by induction:
it holds with equality at $m = 4$ ($\Len{2}(4) = 32 = 2 \cdot 4^2$),
and the recurrence propagates it,
since $\Len{2}(2m) = 6\,\Len{2}(m) + 8m^2 \ge 8m^2 = 2(2m)^2$.

With base $\Len{2}(4) = 32$ and this recurrence for $n \ge 8$,
we obtain the following.

\begin{theorem}
\label{thm:2-of-n}
For $n = 2^i$ with $i \ge 2$,
the Demaine-perspective binary-splitting construction with Huffman commutator trees
produces a solution to the \kn{2}{n} picture-hanging puzzle (Demaine convention) of length at most
\begin{equation}
\label{eq:L2}
  \Len{2}(n) \;=\; \tfrac{8}{3}\,n^{\log_2 6} - 4\,n^2.
\end{equation}
In particular, its length is $O(n^{\log_2 6})$ with $\log_2 6 \approx 2.585$.
\end{theorem}

\begin{proof}
The recurrence $\Len{2}(n) = 6\,\Len{2}(n/2) + 2n^2$ with $\Len{2}(4) = 32$
is solved by elementary means.
Setting $\ell_i = \Len{2}(2^i)$, the recurrence becomes
\[
  \ell_i \;=\; 6\,\ell_{i-1} + 2 \cdot 4^i.
\]
The homogeneous solution is $A \cdot 6^i$.
For a particular solution, try $\ell_i = B \cdot 4^i$:
\[
  B \cdot 4^i \;=\; 6 B \cdot 4^{i-1} + 2 \cdot 4^i \;=\; \tfrac{3B}{2}\cdot 4^i + 2 \cdot 4^i,
\]
giving $B - \tfrac{3B}{2} = 2$, hence $B = -4$.
The general solution is $\ell_i = A \cdot 6^i - 4 \cdot 4^i$.
The base $\ell_2 = 32$ gives $32 = 36A - 64$,
hence $A = \tfrac{96}{36} = \tfrac{8}{3}$.
Substituting $n = 2^i$ converts $6^i = n^{\log_2 6}$ and $4^i = n^2$,
yielding the stated closed form for $\Len{2}(n)$, the symbol count before reduction.
The constructed solution is therefore of length at most this value.
\end{proof}

\begin{remark}
The closed form has been verified by direct computation~\cite{Verhoeff2026Repo};
for instance, $\Len{2}(1024) = \tfrac{8}{3}\cdot 6^{10} - 4 \cdot 4^{10} = 157{,}048{,}832$.
\end{remark}

\begin{remark}[Cancellations]
Because $\Len{2}(n)$ counts symbols before free reduction,
it is only an upper bound on the length of the solution produced.
The three disjuncts have overlapping supports
--- the cross term shares variables with both same-half subproblems ---
so free reductions can occur at the commutator junctions and shorten the word
(at $n = 4$, for instance, a naive ordering already yields length $28$ rather than $32$).
Reordering the disjuncts in the commutator tree,
or replacing a subproblem solution by an equivalent one
under the symmetries of Section~\ref{sec:search},
opens further cancellations;
the gains affect the constants but not the exponent $\log_2 6$.
By contrast, Wästlund's binary construction (Section~\ref{sec:n-2-of-n})
keeps disjoint supports throughout,
so no free reduction ever occurs and its length is exact.
\end{remark}

\begin{remark}[Comparison]
Wästlund's deterministic divide-and-conquer
gives quasi-polynomial length $n^{O(\log n)}$ for the general \kn{k}{n} puzzle~\cite{Waestlund2021};
the implicit exponent for $k=2$ via that construction is not closed-form.
Theorem~\ref{thm:2-of-n} gives an explicit closed-form bound of degree $\log_2 6$,
with all constants known exactly.
\end{remark}

\section{Polynomial construction for \kn{(n-2)}{n}}
\label{sec:n-2-of-n}

The cases $k = n$ (length exactly $n$) and
$k = n-1$ (length exactly $2n$, Proposition~\ref{prop:k-equals-n-1}) are trivial.
The first non-trivial case in the co-rank family is $k = n-2$.
Switching to Wästlund's convention, where this is the \kn{3}{n} puzzle,
lets us apply Wästlund's binary-split recursion with constant per-level branching
--- the regime where balanced splitting is most efficient.

We use Wästlund's recursion, with optimal building blocks for the smaller cases:
the \kn{1}{m} subproblems (Wästlund's convention, equivalently \kn{m}{m} in Demaine's convention)
use the length-$m$ solution of Proposition~\ref{prop:k-equals-n},
and the \kn{2}{m} subproblems (Wästlund's convention, equivalently \kn{(m-1)}{m} in Demaine's convention)
use the length-$2m$ solution of Proposition~\ref{prop:k-equals-n-1}.
Both building blocks are linear in $m$; Wästlund's general construction is significantly worse in these cases.

\begin{theorem}
\label{thm:n-2-of-n}
For $n = 2^i$ with $i \ge 1$, Wästlund's binary-split construction
--- applied to \kn{3}{n} in Wästlund's convention,
equivalently \kn{(n-2)}{n} in Demaine's convention,
using optimal $k=1$ and $k=2$ subproblem solutions ---
produces a solution of length exactly
\begin{equation}
\label{eq:Ln2}
  \Len{n-2}(n) \;=\; 6n\log_2(n/2) \;=\; 6n(\log_2 n - 1).
\end{equation}
In particular, $\Len{n-2}(n) = O(n\log n)$.
\end{theorem}

\begin{proof}
The recurrence is $L(2^i) = 2\,L(2^{i-1}) + 2(2(2^i+2^{i-1}))$ with $L(2) = 0$,
which simplifies to $L(2^i) = 2\,L(2^{i-1}) + 6 \cdot 2^i$.
Setting $\ell_i = L(2^i)$, we have $\ell_i = 2\ell_{i-1} + 6 \cdot 2^i$.
The homogeneous solution is $A \cdot 2^i$; since $f(n) = 6n$ matches the homogeneous form,
the particular solution has the form $B \cdot i \cdot 2^i$:
\[
  B \cdot i \cdot 2^i \;=\; 2 B(i-1) \cdot 2^{i-1} + 6 \cdot 2^i \;=\; B(i-1)\cdot 2^i + 6 \cdot 2^i,
\]
so $Bi = B(i-1) + 6$, giving $B = 6$.
The general solution is $\ell_i = A \cdot 2^i + 6i \cdot 2^i$.
The base $\ell_1 = 0$ gives $0 = 2A + 12$, so $A = -6$.
Hence
\[
  \ell_i \;=\; (6i - 6)\cdot 2^i \;=\; 6(i-1)\cdot 2^i.
\]
With $n = 2^i$ this is $6 n (\log_2 n - 1) = 6n\log_2(n/2)$.
Correctness of the construction
--- that the word produced does solve the \kn{(n-2)}{n} puzzle ---
follows from Wästlund's analysis,
since the construction is his AND-of-ORs structure with disjoint-support discipline at every junction.
As the supports stay disjoint, no free reduction ever occurs, so the length above is exact.
\end{proof}

\subsection*{Summary of regimes}

Combining Theorems~\ref{thm:2-of-n} and~\ref{thm:n-2-of-n} with the trivial cases:

\begin{table}[ht]
  \centering
  \begin{tabular}{l|l|l}
    \emph{Regime} & \emph{Construction length (powers of two)} & \emph{Status} \\
    \hline
    $k = n$ & exactly $n$ & optimal \\
    $k = n-1$ & exactly $2n$ & optimal (Proposition~\ref{prop:k-equals-n-1}) \\
    $k = n-2$ & exactly $6n\log_2(n/2)$ & Theorem~\ref{thm:n-2-of-n} \\
    $k = 1$ & exactly $n^2$ & optimal (Proposition~\ref{prop:k-equals-1}) \\
    $k = 2$ & at most $\tfrac{8}{3}n^{\log_2 6} - 4n^2$ & Theorem~\ref{thm:2-of-n} \\
  \end{tabular}
  \caption{Length results for small $k$ and small $n-k$, for $n$ a power of two.}
  \label{tab:regimes}
\end{table}

The asymmetry between Demaine's small $k$ and small $n-k$ is striking:
our $k=2$ construction reaches word length $O(n^{\log_2 6})$,
while the ``dual'' $k = n-2$ needs only $\Theta(n\log n)$.
The two convention orientations exchange where the polynomial degree is small.

\section{An exponential construction by one-step extension}
\label{sec:extension}

The polynomial construction of Section~\ref{sec:binary-split} is asymptotically efficient,
but at small $n$ its constants are loose:
for $n=4$ it gives length $32$.
We now present an alternative construction that is asymptotically exponential
but produces shorter solutions for small $n$. It also illustrates a different structural principle
--- overlapping variable supports with controlled cancellation ---
and was the path by which the author first reached lengths below $24$ for \kn{2}{4},
motivating the exhaustive search of Section~\ref{sec:search}.
This approach was discovered after the Vierkant workshop while preparing for the MathFest.

\subsection{The recursion}

We construct $\Hang{k}{n}$ from $\Hang{k}{n-1}$ and $\Hang{k-1}{n-1}$,
for any $k$ with $2 \le k \le n-1$;
the case $k = 2$ gives the short \kn{2}{n} solutions we use below.

\begin{proposition}
\label{prop:extension}
Let $\Hang{k}{n-1}$ solve the \kn{k}{(n-1)} puzzle
and $\Hang{k-1}{n-1}$ solve the \kn{(k-1)}{(n-1)} puzzle,
both on the variables $\{1, \ldots, n-1\}$.
Define
\begin{equation}
\label{eq:recursion}
  \Hang{k}{n} \;=\; \Hang{k}{n-1} \,+\, n \,+\, \Hang{k-1}{n-1} \,-\, \Hang{k}{n-1} \,-\, n.
\end{equation}
Then $\Hang{k}{n}$ solves the \kn{k}{n} puzzle on $\{1, \ldots, n\}$.
\end{proposition}

\begin{proof}
Write $A = \Hang{k}{n-1}$ and $B = \Hang{k-1}{n-1}$,
so $\Hang{k}{n} = A + n + B - A - n$,
and let $\ell$ count the removed nails in $\{1, \ldots, n-1\}$.
Removal acts on each part separately, so
$\Hang{k}{n}|_S = A|_S + n|_S + B|_S - A|_S - n|_S$,
where $n|_S$ is $n$ if nail $n$ survives and $0$ if it is removed.
Since $A$ solves \kn{k}{(n-1)} and $B$ solves \kn{(k-1)}{(n-1)},
\[
  A|_S = 0 \iff \ell \ge k,
  \qquad
  B|_S = 0 \iff \ell \ge k-1 .
\]
By the essential-removals reduction (Section~\ref{sec:notation})
it suffices to check that every removal of $k$ nails makes $\Hang{k}{n}$ fall,
and every removal of $k-1$ nails leaves it hanging.

\emph{Fall, $|S| = k$.}
If nail $n$ survives then $\ell = k$, so $A|_S = B|_S = 0$ and $\Hang{k}{n}|_S = n - n = 0$.
If nail~$n$ is removed then $\ell = k-1$, so $B|_S = 0$ and $\Hang{k}{n}|_S = A|_S - A|_S = 0$.

\emph{Hang, $|S| = k-1$.}
If nail $n$ survives then $\ell = k-1$, so $A|_S \ne 0$ and $B|_S = 0$,
leaving the commutator $\Hang{k}{n}|_S = \Comm{A|_S}{n}$,
nonzero by Lemma~\ref{lem:disjoint} since $\{1, \ldots, n-1\}$ and $\{n\}$ are disjoint.
If nail $n$ is removed then $\ell = k-2$, so $B|_S \ne 0$,
leaving $\Hang{k}{n}|_S = A|_S + B|_S - A|_S$, a conjugate of $B|_S$;
a conjugate solves the same problem (Section~\ref{sec:notation}), so it is nonzero.
This last case is the only essential removal in which the non-disjoint sum $B - A$ survives;
the surrounding conjugation makes it harmless.
\end{proof}

The length of $\Hang{k}{n}$ via~\eqref{eq:recursion}
is at most $2\,|\Hang{k}{n-1}| + |\Hang{k-1}{n-1}| + 2$,
with equality when no free reduction occurs at the internal junctions.
The recursion doubles $|\Hang{k}{n-1}|$ at each step (plus an additive cost),
so the length is $\Theta(2^n)$ at worst, exponential in $n$.

\subsection{Cancellations and the small-$n$ values}

We now specialize to $k = 2$.
The construction admits considerable freedom:
$\Hang{1}{n-1}$ can be any \kn{1}{(n-1)} solution,
and the orientations of its subcommutators
(e.g.\ choosing $\Comm{a}{b}$ versus $\Comm{b}{a}$)
can be selected to align the suffix of $\Hang{1}{n-1}$ with the prefix of $-\Hang{2}{n-1}$,
producing free reductions at the junction.

\subsubsection*{The case $n=4$}

A minimal \kn{2}{3} solution is
\[
  \Hang{2}{3} \;=\; 1+2+3-1-2-3,
\]
of length $6$. Its inverse $-\Hang{2}{3} = 3+2+1-3-2-1$ starts with $3+2+1$.

The junction in~\eqref{eq:recursion} lies between $\Hang{1}{3}$ and $-\Hang{2}{3}$;
to maximize cancellation, $\Hang{1}{3}$ should end in $-1-2-3$,
matching $3+2+1$ in reverse.
The choice
\[
  \Hang{1}{3} \;=\; \Comm{\Comm{2}{1}}{3} \;=\; 2+1-2-1+3+1+2-1-2-3,
\]
of length $10$, achieves this.
Substituting into~\eqref{eq:recursion} and indicating the junction:
\[
  \Hang{2}{4} \;=\; \Hang{2}{3} + 4 + \underbrace{\Hang{1}{3}}_{\text{ends } -1-2-3} + \underbrace{(-\Hang{2}{3})}_{\text{starts } 3+2+1} - 4.
\]
At the junction three pairs cancel
()$-3+3$, then $-2+2$, then $-1+1$)
and the result has length~$18$:
\[
  \Hang{2}{4} \;=\; 1+2+3-1-2-3+4+2+1-2-1+3+1+2-3-2-1-4.
\]

\subsubsection*{Other orientation choices}

How much cancellation occurs is exactly how long a suffix of $\Hang{1}{3}$
matches $-1-2-3$:
\begin{itemize}
\item $\Hang{1}{3} = \Comm{1}{\Comm{2}{3}}$ ends in $-3-2$:
  no cancellation, length $24$
  --- equal to what Wästlund's binary-split recursion produces for \kn{2}{4}.
\item $\Hang{1}{3} = \Comm{\Comm{1}{2}}{3}$ ends in $-1-3$:
  one pair cancels, length $22$.
\item $\Hang{1}{3} = \Comm{1}{\Comm{3}{2}}$ ends in $-2-3$:
  two pairs cancel, length $20$.
\end{itemize}
A clean $24$--$22$--$20$--$18$ ladder by cancellation count.
The length-18 solution was found only after the MathFest.

\subsubsection*{Larger $n$}

The same recursion applied with $n=5$ and using the length-$18$ solution
for $\Hang{2}{4}$ together with a length-$16$ solution for $\Hang{1}{4}$
(the standard balanced commutator $\Comm{\Comm{1}{2}}{\Comm{3}{4}}$)
yields a length-$58$ solution before cancellations and
a length-$54$ solution with cancellations at the recursion's junction.
This is to be compared with Wästlund's length-$66$ for \kn{2}{5}
via his binary-split recursion, and the unknown true minimum.

The recursion remains exponential in $n$,
so it is not asymptotically competitive with the polynomial construction of Section~\ref{sec:binary-split}.
Its value is in providing concretely short explicit solutions at small~$n$,
and in suggesting
--- by the gap between its length-$18$ result for $n=4$ and the construction's $32$ ---
that exhaustive search at $n=4$ might find substantially shorter solutions still.

\subsection{Concrete \kn{2}{4} solutions}
\label{sec:concrete-2-of-4}

Combining the constructions of this section with binary splitting
(Section~\ref{sec:binary-split}) and exhaustive search
(Section~\ref{sec:search}),
we can survey \kn{2}{4} solutions found so far, by decreasing length.

\begin{center}
\begin{tabular}{r|l}
\emph{Length} & \emph{Construction} \\
\hline
$80$ & Demaine et al.~\cite{Demaine2014}, as listed (unreduced count) \\
$68$ & same expression, freely reduced \\
$58$ & Demaine's tree, leaves reoriented for more cancellations \\
$54$ & further reordering, mixed subtree shapes \\
$52$ & Demaine's tree, redundant leaf $3+4$ dropped \\
$44$ & Demaine's tree reorganized into clean 2-of-3 atoms \\
$32$ & binary splitting, good Huffman order ($=\Len{2}(4)$, unreduced) \\
$24$ & extension, no junction cancellation (= W\"astlund's binary-split) \\
$22$ & extension, $\Hang{1}{3} = \Comm{\Comm{1}{2}}{3}$ (1 cancellation) \\
$20$ & extension, $\Hang{1}{3} = \Comm{1}{\Comm{3}{2}}$ (2 cancellations) \\
$18$ & extension, $\Hang{1}{3} = \Comm{\Comm{2}{1}}{3}$ (3 cancellations) \\
$16$ & exhaustive search, optimal $w_1, w_2$ (Section~\ref{sec:search} below)
\end{tabular}
\end{center}

A few of these warrant comment.

\paragraph{Reoriented leaves of Demaine's tree.}
Lengths $58$ and $54$ are achieved by reversing the order of generators
within some leaves $i + j \mapsto j + i$,
and (for $54$) varying the within-side tree shape:
\begin{align*}
58: &\quad \Comm{\Comm{1+2}{\Comm{3+1}{4+1}}}{\Comm{3+2}{\Comm{2+4}{3+4}}}, \\
54: &\quad \Comm{\Comm{\Comm{2+1}{3+1}}{3+4}}{\Comm{4+1}{\Comm{3+2}{4+2}}}.
\end{align*}
Both still cover all six pairs and so remain valid \kn{2}{4} solutions
(Remark~\ref{rem:demaine}).
This idea was suggested to me at the Vierkant workshop by one of the assistants.

\paragraph{Huffman ordering of the binary-split disjuncts.}
The three disjuncts for the balanced split $L = \{1,2\}, R = \{3,4\}$
are $1+2$, $3+4$, and the cross term $\Comm{1}{2}+\Comm{3}{4}$,
of lengths $2$, $2$, $8$.
The good Huffman order places the long cross term at depth~$1$,
$\Comm{\Comm{1+2}{3+4}}{\Comm{1}{2}+\Comm{3}{4}}$,
contributing $2\cdot 8 + 4\cdot 2 + 4\cdot 2 = 32$.
The bad order does the reverse,
$\Comm{\Comm{1+2}{\Comm{1}{2}+\Comm{3}{4}}}{3+4}$,
contributing $2\cdot 2 + 4\cdot 8 + 4\cdot 2 = 44$
--- the same length-$44$ as Demaine's clean-atom reorganization below.

\paragraph{From $68$ to $44$ via Demaine.}
Two reductions of Demaine's tree shorten the word.
First, the leaf $3+4$ is redundant:
its minterm $\{3,4\}$ is already supplied by the don't-care collapse of
$\Comm{1+3}{1+4}$ at the removal $\{3,4\}$
(and $\{2,3\}$ similarly by $\Comm{2+4}{3+4}$ at $\{2,3\}$),
so dropping it gives length~$52$.
Second, each inner $\Comm{1+i}{1+j}$ in fact solves $\kof{2}{\{1,i,j\}}$
through its collapse,
so the outer structure rewrites as the clean-spec tree
$\Comm{\Comm{1+2}{\Hang{2}{\{1,3,4\}}}}{\Hang{2}{\{2,3,4\}}}$,
in which each \kn{2}{3} block can be any solution of its specification.
Once a subtree's specification is clean
(its parent node satisfies the separation hypothesis of
Lemma~\ref{lem:combination}),
correctness depends only on the specification, not the realizing word.
Replacing each \kn{2}{3} block above
(a commutator-tree realization of length~$8$)
by the canonical \kn{2}{3} solution $a + b + c - a - b - c$ of length~$6$
(Proposition~\ref{prop:k-equals-n-1})
preserves correctness and brings the length down to~$44$.

\section{Exhaustive search and exact minima}
\label{sec:search}

\subsection*{Parity}

In any solution to the \kn{k}{n} puzzle with $k < n$,
every variable has net exponent zero.
Setting all variables except $j$ to $0$ removes $n-1 \ge k$ of them,
so the resulting expression must vanish;
hence the net exponent of $j$ in any solution is zero.
Since each variable appears with net exponent zero,
it appears an even number of times,
so the total length is even.
Odd lengths are eliminated without search.

\subsection*{Symmetry quotienting}

The search canonicalizes each candidate word by four symmetry classes:
\begin{enumerate}
\item[(S1)] \emph{Nail relabeling} ($S_n$):
  nails are renamed so that their labels appear in order of first occurrence ($1, 2, 3, \ldots$).
\item[(S2)] \emph{Per-nail sign flip} ($(\mathbb{Z}_2)^n$):
  each nail's first occurrence is made positive.
\item[(S3)] \emph{Reversal}:
  reading $w$ backward gives an equivalent solution;
  together with the sign flips of~(S2) this is the wire inversion $w \mapsto -w$.
\item[(S4)] \emph{Cyclic rotation}:
  any cyclic rotation of $w$ gives an equivalent solution.
\end{enumerate}
A word is \emph{canonical} if it is
the lexicographically smallest $(\text{S1}+\text{S2})$-normalized form
among all its cyclic rotations and the cyclic rotations of its reversal.
Each equivalence class contains exactly one canonical representative,
substantially reducing the search space.

\subsection*{Confirmed exact minima}

The following results are established by exhaustive search and
verified by the unit test suite~\cite{Verhoeff2026Repo}.
\begin{itemize}
\item \emph{\kn{2}{3}:} minimum length exactly $6$;
  the unique canonical solution is $1+2+3-1-2-3$
  (confirming Proposition~\ref{prop:k-equals-n-1}).
\item \emph{\kn{1}{3}:} minimum length exactly $10$;
  the unique canonical solution is $1+2-1-2+3+2+1-2-1-3$.
\item \emph{\kn{1}{4}:} minimum length exactly $16$;
  the unique canonical solution is $\Comm{\Comm{1}{2}}{\Comm{3}{4}}$
  (confirming Proposition~\ref{prop:k-equals-1} at $n=4$).
\item \emph{\kn{2}{4}:} minimum length exactly $16$;
  exactly two canonical solutions (see Figure~\ref{fig:2-out-of-4-optimal} and below).
\end{itemize}
\begin{figure}[hbt]
\centering
\includegraphics[height=5cm]{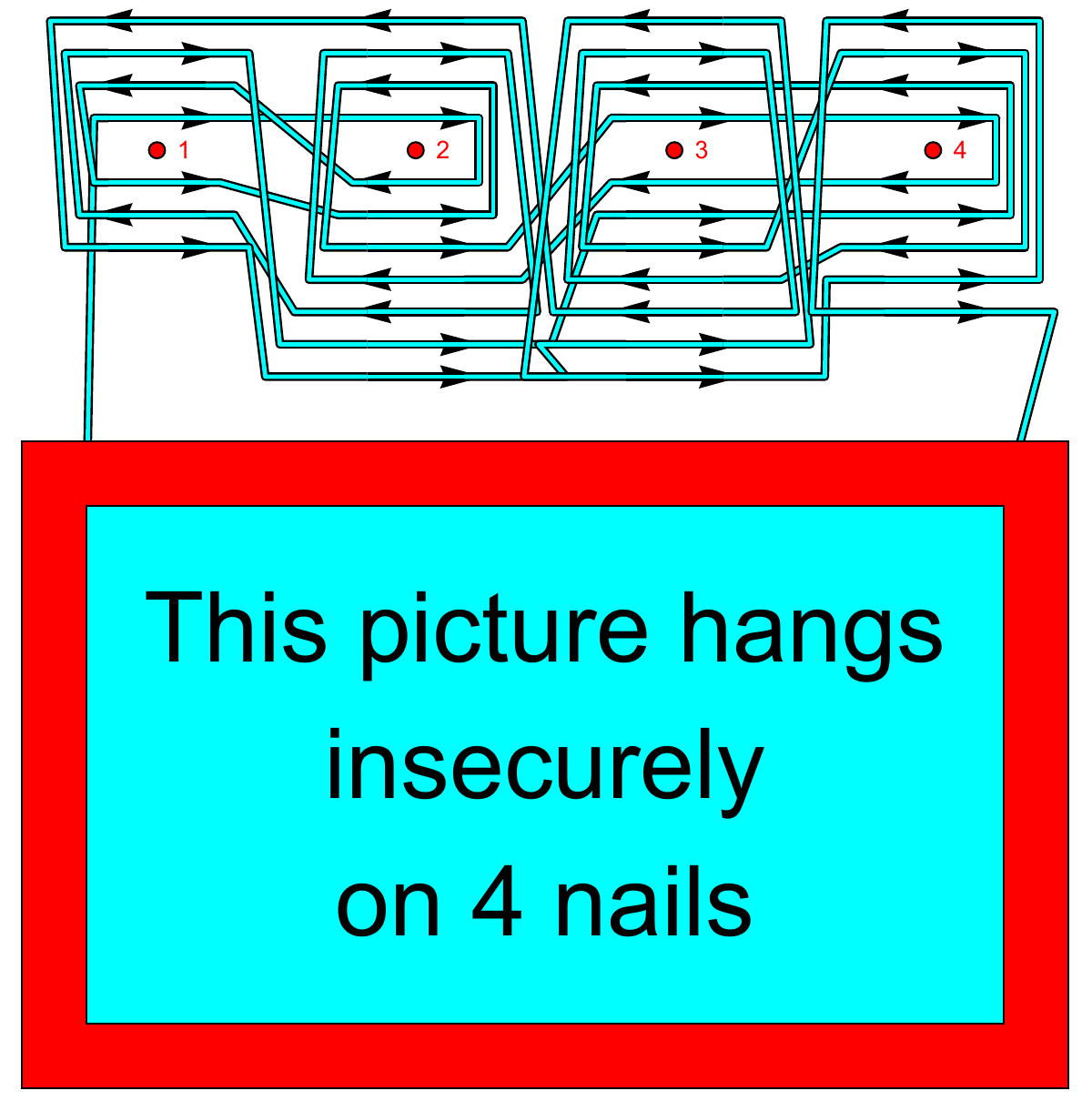}
\caption{An optimal length-16 solution for the \kn{2}{4} picture-hanging puzzle.}
\label{fig:2-out-of-4-optimal}
\end{figure}

\subsection*{Main result: exact minimum for \kn{2}{4}}

No solution to the \kn{2}{4} puzzle exists with length below $16$:
even lengths $2, 4, \ldots, 12$ are eliminated by a fast symmetry-pruned search;
length $14$ is eliminated by a longer dedicated search run (just over two hours in Python);
all odd lengths are eliminated by the parity argument.
At length $16$, the search finds exactly two canonical solutions~\cite{Verhoeff2026Repo}:
\begin{align*}
  w_1 &= 1+2-1-2+3+4+2+1-4-3+4+3-2-1-3-4, \\
  w_2 &= 1+2-1+3-2+4+2+1-4-3+4-2+3-1-3-4.
\end{align*}
Every \kn{2}{4} solution of length $16$ is equivalent to $w_1$ or $w_2$
under the four symmetries above.
The supporting code is available at~\cite{Verhoeff2026Repo},
with unit tests covering the free-group operations,
puzzle property check, and canonical-form symmetries.

For \kn{2}{4}, Theorem~\ref{thm:2-of-n} gives the bound $32$;
exhaustive search reaches the proven optimum~$16$, well below it.
The structural relationship between the asymptotic construction
and the small-$n$ optimum is an open question (Section~\ref{sec:open}).

\section{Workshop realizations: carabiners and Whitehead links}
\label{sec:carabiners}

For workshop purposes,
the constructions are realized on a wooden board fitted with carabiners rather than nails
(see Figure~\ref{fig:carabiner-board}).
A carabiner physically prevents the wire from slipping off,
which is essential when participants handle a partially completed construction:
the picture should fall only when the relevant carabiners are deliberately ``removed''
(opened, and the wire pulled through).
However, this convenience introduces an ambiguity that does not arise with vertical nails under gravity.
\begin{figure}[ht]
  \centering
  \includegraphics[trim=23cm 0 5cm 0,clip,height=5cm]{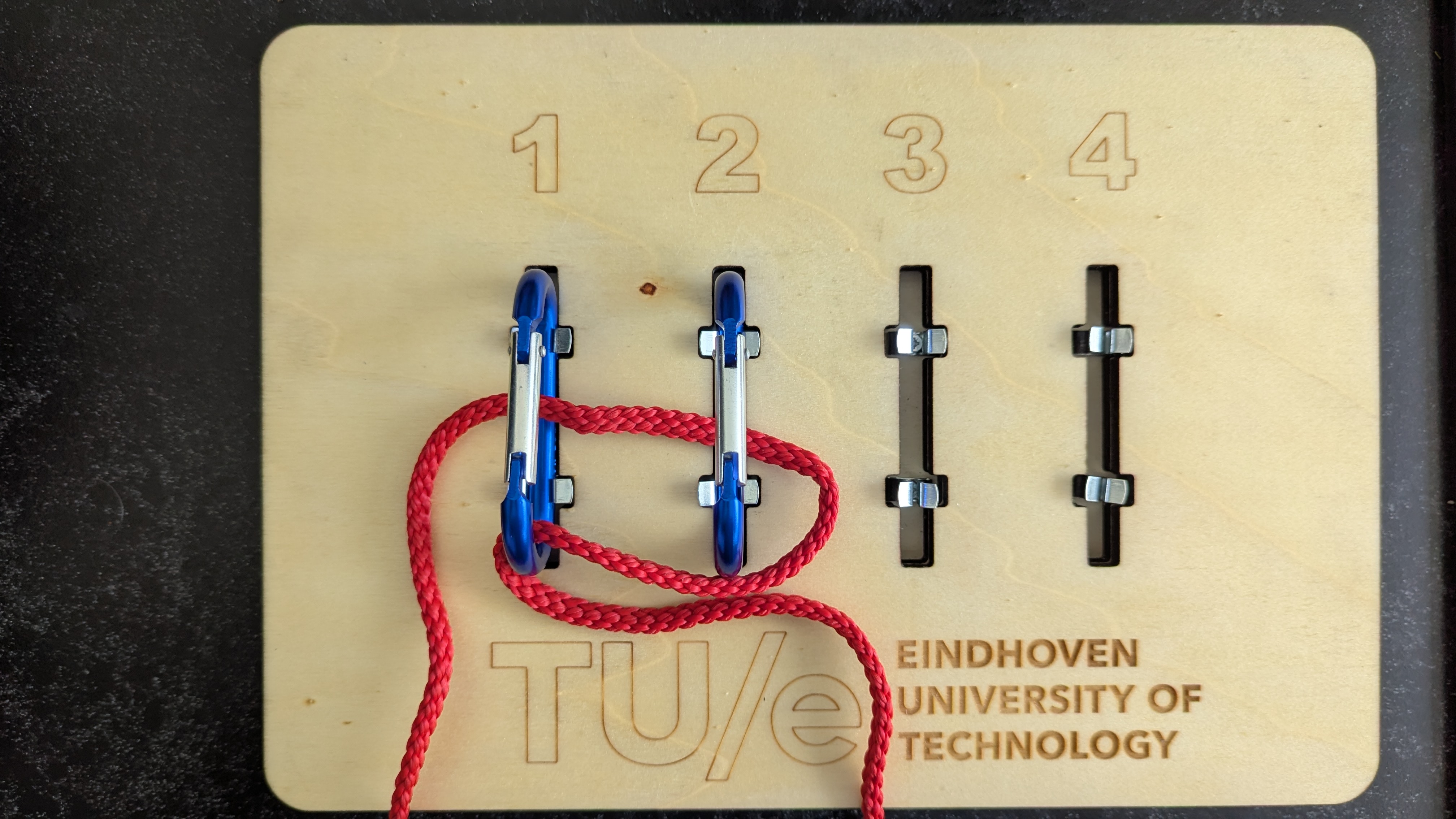}\hfil
  \caption{Hanging $1+2-1$ on board with two carabiners.}
  \label{fig:carabiner-board}
\end{figure}

With nails and gravity,
the wire's vertical hierarchy at each crossing is fixed by physics:
a later strand naturally falls in front of
(or behind, depending on how the wire is paid out) an earlier one.
With carabiners on a flat board, the wire's over/under choice at each crossing is a free parameter,
and incorrect choices produce topologically different links.

\subsection*{The simplest failure: a Whitehead link}

Consider the standard \kn{1}{2} solution $1+2-1-2$.
If the second factor $-2$ is wired with the wrong over/under choice relative to the preceding strands,
the resulting wiring no longer represents the commutator in $F_2$
but instead realizes a Whitehead link with one of the carabiners
(see Figure~\ref{fig:whitehead}).
Removing the other nail then leaves the wire genuinely linked around the remaining carabiner
rather than free to fall.

\begin{figure}[ht]
  \centering
  \includegraphics[trim=30cm 0 47cm 3cm,clip,height=5cm]{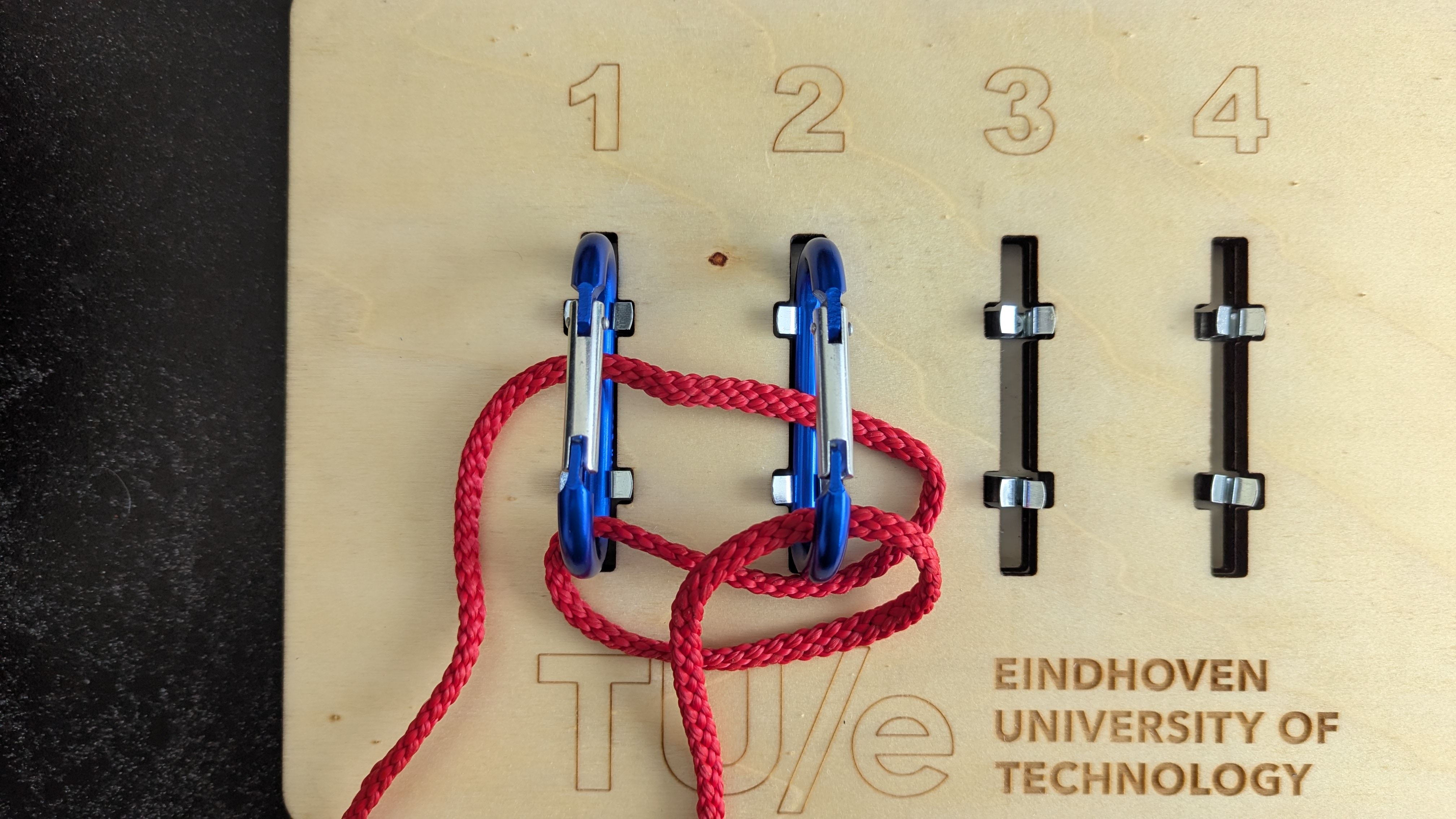}\hfil%
  \includegraphics[trim=30cm 1cm 47.5cm 2cm,clip,height=5cm]{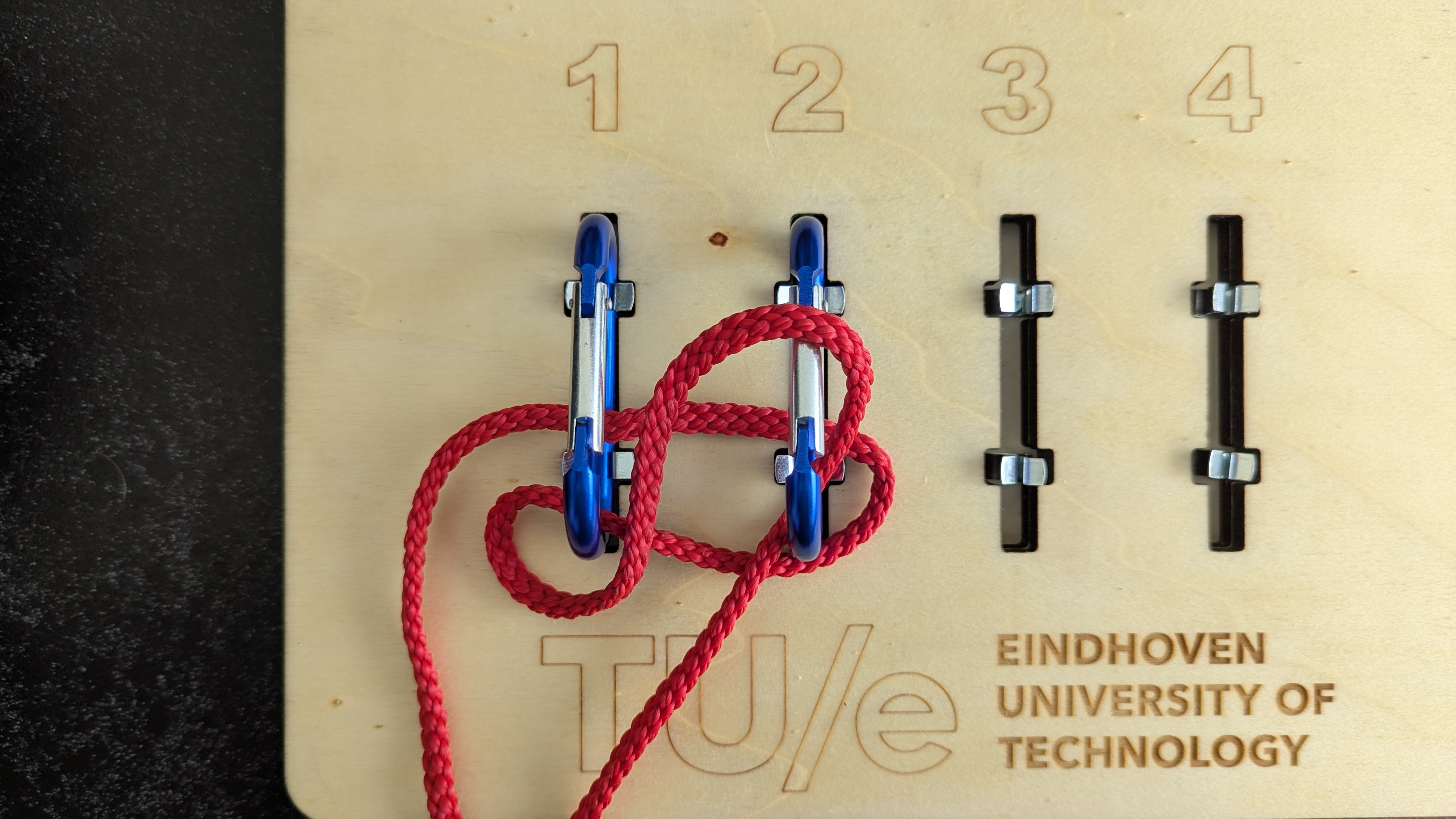}\hfil%
  \includegraphics[trim=40cm 0 42cm 2cm,clip,height=5cm]{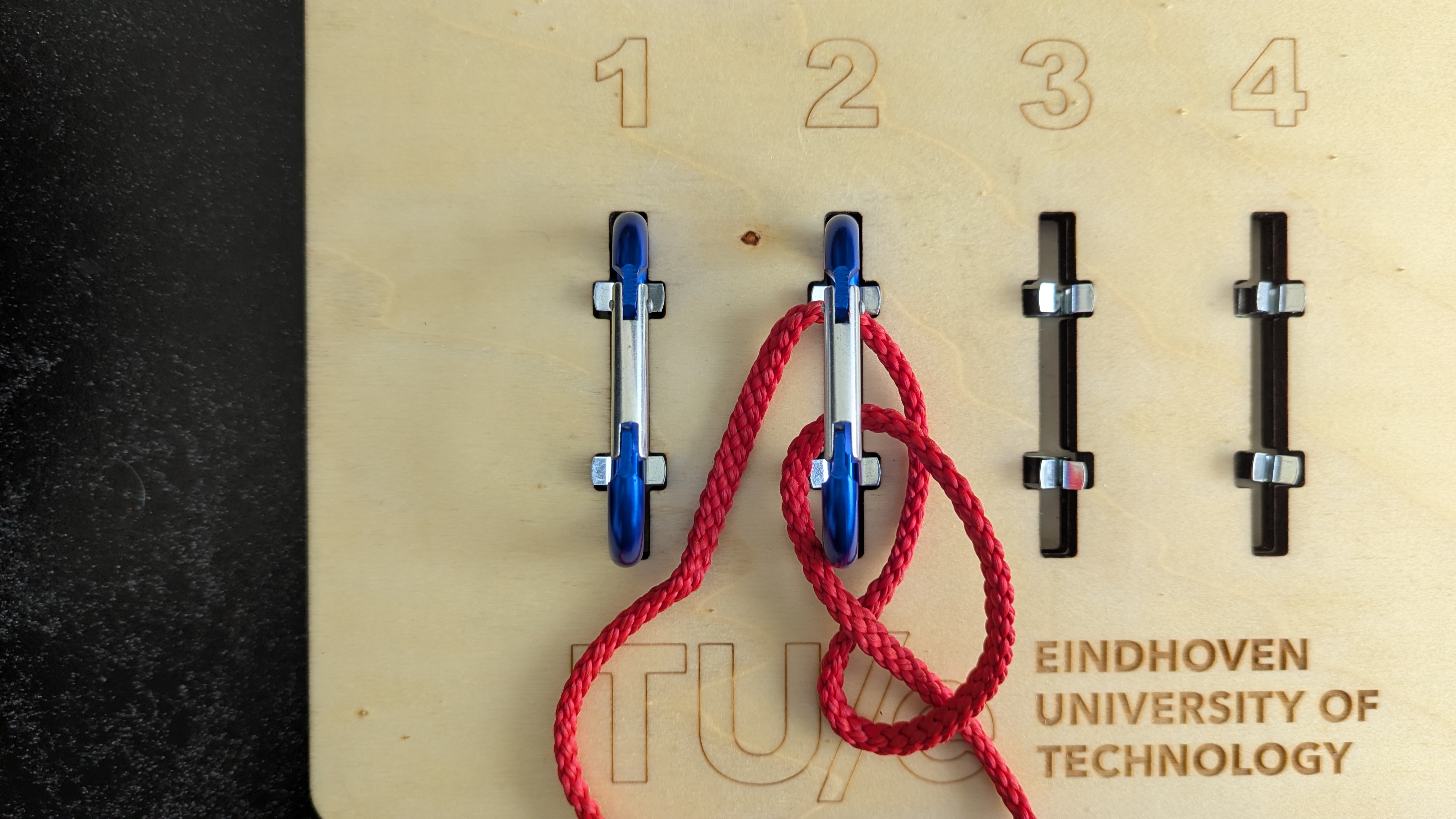}
  \caption{Left: correct hanging $1+2-1-2$,
  in which the strand for $-2$ passes consistently above the previous strands.
  Middle: incorrect over/under choice for~$-2$.
  Right: removing nail~$1$ produces a Whitehead link, which does not fall.}
  \label{fig:whitehead}
\end{figure}

This failure mode was observed during the workshop.
Beyond its pedagogical value
--- participants quickly internalize that the algebra dictates the topology
only if the geometric realization is consistent ---
it is also a reminder that the standard picture-hanging analysis
implicitly assumes a particular planar projection or a gravitational hierarchy.
In the absence of such a constraint,
the same word $w \in F_n$ can correspond to multiple non-equivalent wirings.

\section{Open questions}
\label{sec:open}

Several questions are left open.

First, what is the exact length of the general Demaine-perspective binary-splitting construction (Section~\ref{sec:binary-split})
for fixed $k \ge 3$?
(The construction is correct for every $k$ by Theorem~\ref{thm:correctness};
only its length is at issue.)
For each fixed $k$, Huffman placement yields a divide-and-conquer recurrence $\Len{k}(n) = c_k\,\Len{k}(n/2) + (\text{lower-order terms})$.
The recursive coefficient $c_k$ is governed not by the number of disjuncts ($k+1$)
but only by the two disjuncts that are themselves \kn{k}{(n/2)} subproblems,
namely the $j = 0$ and $j = k$ terms:
Huffman placement keeps these two nearest the root
while the remaining, lower-order disjuncts are pushed deeper.
Computation indicates that $c_k = 6$ for \emph{every} $k \ge 2$
(a balanced tree would instead give the larger coefficient $\approx 2(k+1)$),
so the polynomial exponent saturates at $\log_2 6$ rather than growing with $k$.
The lower-order terms, however, appear to carry polylogarithmic factors in $n$
whose power grows with $k$; pinning down this exact asymptotic is left open.

Second, what are the exact minima for the smallest open cases beyond \kn{2}{4}?
The exhaustive search techniques of Section~\ref{sec:search} reach lengths up to about $16$ in pure Python;
the next case, \kn{3}{5}
(equivalently, \kn{(n-2)}{n} at $n=5$, the first non-trivial case in the co-rank family for odd $n$),
should be within reach with a more efficient implementation.

Third, the leading constant $\tfrac{8}{3}$ in Theorem~\ref{thm:2-of-n}
and the constant $6$ in Theorem~\ref{thm:n-2-of-n} have not been given structural derivations.
They emerge from the recurrence solutions;
their cleanliness suggests they should have direct combinatorial interpretations in the construction.

Fourth, the relationship between the polynomial construction
(Theorem~\ref{thm:2-of-n}, whose bound is $\Len{2}(4) = 32$)
and the exhaustive-search optimum ($16$ for \kn{2}{4}, Section~\ref{sec:search})
is asymptotically obscure.
The optimum at $n=4$ lies well below the construction, even after cancellations;
whether this gap persists, shrinks, or grows with $n$ is unknown.
Computational evidence at $n=5,6,7,8$ would clarify this,
but pushing exhaustive search to those sizes is currently infeasible.

Fifth, the workshop-motivated question (Section~\ref{sec:carabiners}):
can the over/under choice at each crossing be incorporated into the algebraic formalism,
perhaps via braid groups,
in a way that distinguishes the intended commutator from the Whitehead link?
This is more an exposition than a research question,
but a clean treatment would benefit pedagogy.

Sixth, is there a ``puncturing'' or ``shortening'' operation
--- by analogy with coding theory --- that lowers $n$ at fixed $k$?
Removing a nail is the easy move in the wrong direction:
setting nail $n+1$ to $0$ in a \kn{k}{(n+1)} solution yields a \kn{(k-1)}{n} solution,
lowering the threshold along with the nail count.
What is missing is an operation that removes a nail while \emph{preserving} the threshold~$k$,
turning a \kn{k}{(n+1)} solution into a \kn{k}{n} one.
Naively deleting a nail's letters does not preserve the puzzle property;
a non-trivial reduction of this kind would supply a recursion
that decreases $n$ by one at fixed $k$,
complementing the halving recursion of Section~\ref{sec:binary-split}
and possibly narrowing the construction-versus-search gap above.

\section*{Acknowledgements}

The author thanks the organizers and participants
of the workshop at the summer-camp preparation day of the
Stichting Vierkant voor Wiskunde and of the mini lecture at the MathFest 2026
that led to this investigation.
In particular, Jens Heuseveldt assisted enthusiastically at both events and
suggested the first improvements on Demaine et al.'s solution to the \kn{2}{4} picture-hanging puzzle.
Discussions with Claude (Anthropic) helped clarify
the structural comparison between the various construction families and
the derivation of the closed-form recurrences.

\bibliographystyle{plain}

\end{document}

%% file: preamble.tex
\usepackage[utf8]{inputenc}
\usepackage[T1]{fontenc}
\usepackage{lmodern}
\usepackage{amsmath}
\usepackage{amssymb}

\newcommand{\Comm}[2]{[#1, #2]}             
\newcommand{\kn}[2]{$#1$-out-of-$#2$}       
\newcommand{\kof}[2]{#1\text{-of-}#2}       
\newcommand{\Hang}[2]{h_{#1,#2}}            
\newcommand{\Len}[1]{L_{#1}}                

%% file: better-picture-hanging.bbl
\begin{thebibliography}{9}

\bibitem{Demaine2014}
E.~D. Demaine, M.~L. Demaine, Y.~N. Minsky, J.~S.~B. Mitchell, R.~L. Rivest, M.~Pătrașcu.
``Picture-Hanging Puzzles.''
\emph{Theory of Computing Systems}, 54(4):531--550, 2014.
\href{https://arxiv.org/abs/1203.3602}{arXiv:1203.3602}.

\bibitem{GartsideGreenwood2007}
P.~Gartside, S.~Greenwood.
````Brunnian Links.''
\emph{Fundamenta Mathematicae}, 193:259--276, 2007.

\bibitem{Spivak1997}
A.~Spivak.
``Brainteasers B 201: Strange Painting.''
\emph{Quantum}, p.~13, May/June 1997.
Solution on p.~60, Figure~5.

\bibitem{Taylor2002}
C.~Lusby Taylor.
Contribution to ``Hanging a Picture'' Thread.
\emph{mathpuzzle.com}, circa~2002.
\url{https://www.mathpuzzle.com/hangingpicture.htm}.

\bibitem{Verhoeff2026}
T.~Verhoeff.
``Picture-Hanging Puzzle.''
\emph{Wolfram Demonstration Projects}, 2026.
\url{https://demonstrations.wolfram.com/PictureHangingPuzzles/}.

\bibitem{Verhoeff2026Repo}
T.~Verhoeff.
``'Picture-Hanging Puzzles: Search and Verification Code.''
Accessed 25~May 2026.
\url{https://gitlab.tue.nl/t-verhoeff-software/picture-hanging-puzzles/}.

\bibitem{Waestlund2021}
J.~Wästlund.
``'Faulty picture-hanging improved.''
\href{https://arxiv.org/abs/2102.00984}{arXiv:2102.00984}, 2021.

\end{thebibliography}
